\documentclass[12pt]{article}

\usepackage{amsfonts,graphicx,amsmath,amssymb,amsthm,color,nicefrac,enumerate, layout, bbm, listings,  hyperref}
\usepackage[latin1]{inputenc}

\newcommand{\R}{\mathbb{R}}

\newcommand{\N}{\mathbb{N}}

\newcommand{\lf}{\lfloor}
\newcommand{\rf}{\rfloor}
\newcommand\numberthis{\addtocounter{equation}{1}\tag{\theequation}}

\newtheorem{theorem}{Theorem}[section]
\newtheorem{lemma}[theorem]{Lemma}

\newtheorem{cor}[theorem]{Corollary}

\let\inf\relax \DeclareMathOperator*\inf{\vphantom{p}inf}

\let\lim\relax \DeclareMathOperator*\lim{\vphantom{p}lim}

\begin{document}

\title{Existence and uniqueness properties\\
	 for solutions of a class of Banach\\
	  space valued evolution equations}

\author{Arnulf Jentzen, Sara Mazzonetto, and Diyora Salimova
}

\maketitle

\begin{abstract}
In this note we provide a self-contained proof of an existence and uniqueness result for a class of Banach space valued evolution equations with an additive forcing term. The  framework of our abstract result includes, for example, finite dimensional ordinary differential equations (ODEs), semilinear deterministic partial differential equations (PDEs), as well as certain additive noise driven stochastic partial differential equations (SPDEs) as special cases. The  framework of our general result assumes somehow mild regularity conditions on the involved semigroup and also allows the involved semigroup operators to be nonlinear. The techniques used in the proofs of our results are essentially well-known in the relevant literature. The  contribution of this note is to provide a rather general existence and uniqueness result which covers several situations as special cases and also to provide a self-contained proof for this existence and uniqueness result.
\end{abstract}

\tableofcontents

\section{Introduction}

In this paper we prove a general existence and uniqueness result for a class of Banach space valued evolution equations with an additive forcing term.   The  framework of our abstract result is quite general and includes, for instance, finite dimensional ordinary differential equations (ODEs), semilinear deterministic partial differential equations (PDEs), and  certain additive noise driven stochastic partial differential equations (SPDEs) as special cases. Moreover, the framework of our result assumes only mild regularity conditions on the involved semigroup and also allows the involved semigroup operators to be nonlinear. The tools used in the proofs of our results are basically well-known in the relevant literature. The main contribution of this paper is to 
provide an existence and uniqueness result which can be applied to different evolution equations 
and additionally, to present a self-contained proof for this existence and uniqueness result.

To illustrate the main result of this article, Corollary~\ref{cor:unique} in Section~\ref{sec:main_result} below, in more detail, we now present in the following theorem, Theorem~\ref{thm:intro} below, a special case of our main result.

\begin{theorem}
\label{thm:intro}
Let $(V, \left\|\cdot\right\|_V)$  and $(W,\left\|\cdot\right\|_W)$
be  separable $\mathbb{R}$-Banach spaces,
let $T \in (0, \infty)$,  $F \in C(V, W)$, $o \in C([0,T],V)$,  
let $S \colon (0,T) \to L(W,V)$ be a $\mathcal{B}((0,T))\slash \mathcal{B}(L(W,V))$-measurable function, 
let $\mathcal{S} \colon [0,T] \to L(V)$ be a $\mathcal{B}([0,T])\slash \mathcal{B}(L(V))$-measurable function, 
and assume for all $r \in [0,\infty)$, $t_1 \in [0,T)$, $t_2\in (0,T-t_1)$, 
$v \in V$ that 
$([0,T] \ni t  \mapsto \mathcal{S}_t v \in V) \in C([0,T],V)$,
$S_{t_1+t_2}=\mathcal{S}_{t_1} S_{t_2}$,  and
$ \sup \big( \big\{\frac{\|F(v)-F(w)\|_W}{\|v-w\|_V} \colon v, w \in V, v \neq w, \|v\|_V + \|w\|_V \leq r\big\} \cup \{0\}  \big) + 
\inf_{\alpha, \rho \in (0,1)} \sup_{s \in (0,T), t \in (s,T)}  [s^{\alpha} (\|S_s\|_{L(W,V)} +
\|S_{t} - S_{s} \|_{L(W,V)} |t-s|^{-\rho})] < \infty$.
Then there exists a unique convex set $J \subseteq [0,T]$ with $\{0\} \subsetneq J$ such that
\begin{enumerate}[(i)]
\item there exists a unique  $x \in C(J, V)$  which satisfies for all $t \in J$ that 
\begin{align}
\int_0^t \|S_{t-s} \, F(x_s) \|_V \, ds < \limsup_{s \nearrow \sup(J)} \left[\frac{1}{(T-s)}+ \|x_s\|_V\right]\!= \infty
\end{align}
and $x_t = \int_0^t S_{t-s} \, F(x_s) \, ds + o_t$, 
\item for all convex sets $I \subseteq [0,T]$  and all  $y \in C(I, V)$  with $I \supseteq J $,
$\forall \, t \in I \colon \int_0^t \|S_{t-s} \, F(y_s) \|_V \, ds < \infty$, and 
$ \forall \, t \in I \colon  y_t = \int_0^t S_{t-s} \, F(y_s) \, ds + o_t$ it holds that $x=y$, and
\item for all  $y \in C(J, V)$   with $\forall \, t \in J \colon \int_0^t \|S_{t-s} \, F(y_s) \|_V \, ds < \infty$,  
$ \forall \, t \in J \colon  y_t = \int_0^t S_{t-s} \, F(y_s) \, ds + o_t$, and $\limsup_{s \nearrow \sup(J)} \|y_s\|_V < \infty$  it holds that $J = [0, T]$.
\end{enumerate}
\end{theorem}

We note that Theorem~\ref{thm:intro} is an immediate consequence of Corollary~\ref{cor:unique} below. 
Existence and uniqueness for evolution equations have been extensively studied in the literature
(see,  e.g., 
Lions~\cite{lions1969quelques}, 
Weissler~\cite{weissler1979semilinear},
Da Prato \& Zabczyk~\cite{dz92,dz96},
Brze{\'z}niak~\cite{brzezniak1997stochastic},
Cazaneve \& Haraux~\cite{cazenave1998introduction},
Sell \& You~\cite{sy02}, 
Van Nerven, Veraar, \& Weis~\cite{Neerven2008,Neerven2012},
Brezis~\cite{Brezis2011}, 
Jentzen \& Kloeden~\cite{JentzenKloeden2011},
Lunardi~\cite{lunardi2012analytic},
Pazy~\cite{pazy2012semigroups}, 
Roub{\'\i}{\v{c}}ek~\cite{roubivcek2013nonlinear},
E et al.~\cite{EJentzenShen2016}
and the references mentioned therein).

The remainder of this article is organized as follows. In Section~\ref{sec:complete} we  recall in 
Lemma~\ref{lem:complete},
Lemma~\ref{lem:closed},
Corollary~\ref{cor:complete}, and Corollary~\ref{cor:complete:R}
some well-known facts about the completeness of certain function spaces. For completeness we also present proofs for Lemma~\ref{lem:complete},
Lemma~\ref{lem:closed},
Corollary~\ref{cor:complete}, and Corollary~\ref{cor:complete:R}.
Corollary~\ref{cor:complete:R} is used in the proof of our existence and uniqueness result in Theorem~\ref{thm:existence} in Section~\ref{sec:local} below
to ensure that the considered function space is complete so that we are in the position to apply the Banach fixed-point theorem. 
In Section~\ref{sec:measurability} we recall in Lemma~\ref{lem:reg:w} and Lemma~\ref{lem:regularity} some well-known facts on the Bochner integrability of certain functions involving suitable semigroups. 
In Section~\ref{sec:continuity} we present in Corollary~\ref{cor:mild.continuity}  an elementary result about the continuity  of mild solutions of certain nonlinear evolution equations.
In  Section~\ref{sec:perturbation} we establish in Lemma~\ref{lemma:perturbation}  a suitable perturbation estimate for mild solutions of nonlinear evolution equations. Lemma~\ref{lemma:perturbation} is an appropriate extended version of Andersson et al.~\cite[Proposition~2.7]{Andersson2015} and 
Jentzen \& Kurniawan~\cite[Corollary~3.1]{JentzenKurniawan2015}.
In Section~\ref{sec:unique}  we present in Corollary~\ref{cor:local:unique} an elementary fact about uniqueness of mild solutions of certain nonlinear evolution equations. 
In Sections~\ref{sec:local} and~\ref{sec:main_result} we employ well-known techniques from the literature on evolution equations to establish in Theorem~\ref{thm:existence}, Theorem~\ref{thm:unique}, and 
Corollary~\ref{cor:unique} the unique local existence of maximal mild solutions of the considered nonlinear evolution equations.

\section{Complete function spaces}
\label{sec:complete}

In this section we  recall in 
Lemma~\ref{lem:complete},
Lemma~\ref{lem:closed},
Corollary~\ref{cor:complete}, and Corollary~\ref{cor:complete:R}
some well-known facts about the completeness of certain function spaces. For completeness we also present proofs for Lemma~\ref{lem:complete},
Lemma~\ref{lem:closed},
Corollary~\ref{cor:complete}, and Corollary~\ref{cor:complete:R} in this section.

\begin{lemma}
\label{lem:complete}
Let $X$ be a non-empty set, 
let $(E, d)$ be a complete metric space, 
let $\mathcal{E}$ be the set 
given by 
\begin{equation}
  \mathcal{E}= \left\{
    f \colon X \to E \colon 
    \inf\nolimits_{e \in E} 
    \sup\nolimits_{x \in X} d(f(x),e)<\infty
  \right\}
 \! , 
\end{equation}
and let 
$ \delta \colon \mathcal{E} \times \mathcal{E} \to [0,\infty) $ 
be the function which satisfies 
for all $f, g \in \mathcal{E}$ that $\delta(f,g) = \sup_{x \in X}d(f(x),g(x))$. 
Then it holds that the pair $(\mathcal{E}, \delta)$ is a complete metric space.
\end{lemma}
\begin{proof}[Proof of Lemma~\ref{lem:complete}]
First, note that for all $f, g, h \in \mathcal{E}$ it holds that $((\delta(f,g)=0)  \Leftrightarrow (f=g))$, $\delta(f,g) = \delta(g,f)$, and
\begin{align}
\begin{split}
\delta(f,g)+ \delta(g,h) &= \sup_{x \in X} d(f(x),g(x)) + \sup_{x \in X} d(g(x),h(x)) \\
& \geq \sup_{x \in X} \big[ d(f(x),g(x)) + d(g(x),h(x)) \big] \\
& \geq \sup_{x \in X} d(f(x),h(x)) = \delta(f,h).
\end{split}
\end{align}
This proves that the pair $(\mathcal{E}, \delta)$ is a metric space. It thus remains to prove that $(\mathcal{E},\delta)$ is complete. For this let $(f_n)_{n \in \N = \{1, 2, 3, \ldots\}} \subseteq \mathcal{E}$ be a Cauchy sequence in $(\mathcal{E}, \delta)$. 
This assures that  for all $ x \in X $ it holds that  $(f_n(x))_{n \in \N} \subseteq E $ 
is a Cauchy sequence in $(E,d)$. The assumption that $(E, d)$ is a complete metric space hence ensures that there exists a function $g \colon X \to E$ such that for all $x \in X$ it holds that 
\begin{align}
\label{eq:complete:g}
\limsup_{n \to \infty} d(f_n(x),g(x))=0.
\end{align}
This and the assumption that $(f_n)_{n \in \N} \subseteq \mathcal{E}$ is a Cauchy sequence establish that 
\begin{align}
\label{eq:complete}
\begin{split}
0 &= \inf_{N \in \N }  \sup_{n,m \in \N \cap [N, \infty)} \sup_{x \in X} d(f_n(x), f_m(x)) 
\\ &
= \inf_{N \in \N} \sup_{n \in \N \cap [N, \infty)} \sup_{x \in X} \sup_{m \in \N \cap[N, \infty)} d(f_n(x), f_m(x))
\\
& \geq \inf_{N \in \N} \sup_{n \in \N \cap [N, \infty)} \sup_{x \in X} \lim_{m \to \infty} d(f_n(x),f_m(x)) 
\\&
= \inf_{N \in \N} \sup_{n \in \N \cap [N, \infty)} \sup_{x \in X}  d(f_n(x),g(x))
\\
& = \limsup_{n \to \infty} \sup_{x \in X}  d(f_n(x),g(x)) .
\end{split}
\end{align}
This implies that there exists a natural number $k \in \N$ such that 
\begin{align}
\sup_{x \in X}  d(f_k(x),g(x)) \leq 1.
\end{align}
The fact that $f_k \in \mathcal{E}$ therefore proves that
\begin{align}
\begin{split}
\inf_{e \in E} \sup_{x \in X} d(g(x), e) &\leq \inf_{e \in E} \sup_{x \in X} \big[ d(g(x), f_k(x)) + d(f_k(x),e) \big] \\
&\leq   \sup_{x \in X}  d(g(x), f_k(x)) + \inf_{e \in E} \sup_{x \in X}  d(f_k(x),e) \\
& \leq 1+ \inf_{e \in E} \sup_{x \in X}  d(f_k(x),e) < \infty.
\end{split}
\end{align}
This ensures that $g \in \mathcal{E}$. Next note that \eqref{eq:complete} assures that $\limsup_{n \to \infty} \delta(f_n,g)=0$. 
The proof of Lemma~\ref{lem:complete} is thus completed.
\end{proof}

\begin{lemma}
\label{lem:closed}
Let $(X, \mathcal{X})$ be a topological space, let $(E, d)$ be a metric space, let $g \colon X \to E$ be a function, 
and let $f_n \in C(X,E)$, $n \in \N$, satisfy that $\limsup_{n\to \infty}  \sup( \{ d(f_n(x), g(x))  \colon x \in X\} \cup \{0\})=0$. 
Then $g\in C(X,E)$.
\end{lemma}
\begin{proof}[Proof of Lemma~\ref{lem:closed}]
Throughout this proof assume w.l.o.g.\ that $X \neq \emptyset$. 
Observe that the assumption that $\limsup_{n\to \infty} \sup_{x \in X} d(f_n(x), g(x))=0$
ensures that for every $\varepsilon \in (0, \infty)$ there exists
a natural number $N_{\varepsilon} \in \N$ such that for all $x \in X$, $n \in (\N \cap [N_{\varepsilon}, \infty))$ it holds that
\begin{align}
\label{eq:closed}
d(f_n(x),g(x)) < \tfrac{\varepsilon}{3}.
\end{align}
Moreover, note that the assumption that $ \forall\, n \in \N \colon f_n \in C(X,E)$ implies that for all $x \in X$, $\varepsilon \in (0, \infty)$ there exists a set $A_{x, \varepsilon} \in \mathcal{X}$ with $x \in A_{x, \varepsilon}$ such that for all $y \in A_{x, \varepsilon}$ it holds that
\begin{align}
d(f_{N_{\varepsilon}}(x), f_{N_{\varepsilon}}(y)) < \tfrac{\varepsilon}{3}.
\end{align}
This  and \eqref{eq:closed} prove that
for all $x \in X$, $\varepsilon \in (0, \infty)$, $y \in A_{x, \varepsilon}$  it holds that
\begin{align}
\begin{split}
d(g(x),g(y)) &\leq d(g(x), f_{N_{\varepsilon}}(x)) + d(f_{N_{\varepsilon}}(x), f_{N_{\varepsilon}}(y)) + d(f_{N_\varepsilon}(y),g(y))\\ &<\tfrac{\varepsilon}{3}+ \tfrac{\varepsilon}{3} + \tfrac{\varepsilon}{3} = \varepsilon.
\end{split}
\end{align}
Hence, we obtain that $g \in C(X,E)$. The proof of Lemma~\ref{lem:closed} is thus completed.
\end{proof}

The next result, Corollary~\ref{cor:complete}, follows directly from Lemma~\ref{lem:complete} and Lemma~\ref{lem:closed}. 
\begin{cor}
\label{cor:complete}
Let $(X, \mathcal{X})$ be a topological space, 
let $(E, d)$ be a complete metric space, 
let $\mathcal{E}$ be the set given by  
\begin{equation}
  \mathcal{E}= 
  \left\{
    f \in C(X, E) \colon \forall \, e \in E \colon \sup\nolimits	_{x \in X} d(f(x),e)<\infty
  \right\}\!,
\end{equation}
and let $\delta \colon \mathcal{E} \times \mathcal{E} \to [0,\infty)$ be the function which satisfies for all $f, g \in \mathcal{E}$ that $\delta(f,g) = \sup(\{d(f(x),g(x)) \colon x \in X\} \cup \{0\})$. Then it holds that the pair $(\mathcal{E}, \delta)$ is a complete metric space.
\end{cor}

The next result, Corollary~\ref{cor:complete:R}, is an immediate consequence of Corollary~\ref{cor:complete}.

\begin{cor}
\label{cor:complete:R}
Let $(X, \mathcal{X})$ be a topological space, let $(E, d)$ be a complete metric space, let $e \in E$, $R \in [0, \infty)$, 
let $\mathcal{E}$ be the set given by  
\begin{equation}
  \mathcal{E}= 
  \left\{f \in C(X, E) \colon  \sup\nolimits_{x \in X} d(f(x),e) \leq R
  \right\}\!,
\end{equation}
and let $\delta \colon \mathcal{E} \times \mathcal{E} \to [0,\infty)$ be the function which satisfies for all $f, g \in \mathcal{E}$ that $\delta(f,g) = \sup(\{d(f(x),g(x)) \colon x \in X\} \cup \{0\}) $.
Then it holds that the pair $ ( \mathcal{E}, \delta ) $ 
is a complete metric space.
\end{cor}

\section{Measurability properties}
\label{sec:measurability}

In this section we recall in Lemma~\ref{lem:reg:w} and Lemma~\ref{lem:regularity} some well-known facts on the Bochner integrability of certain functions involving suitable semigroups.

\begin{lemma}
\label{lem:reg:w}
Let $(V, \left\|\cdot\right\|_V)$ 
be a separable $\R$-Banach space,
let $(W,\left\|\cdot\right\|_W)$ 
be an $\R$-Banach space, let $T \in (0, \infty)$, $w \in W$, and let
$S \colon (0,T) \to L(W,V)$ be a $\mathcal{B}((0,T)) \slash \mathcal{B}(L(W,V))$-measurable function. 
Then it holds that the function 
\begin{equation}
  (0,T) \ni s \mapsto S_{T-s} \, w \in V
\end{equation}
is strongly $\mathcal{B}((0,T)) \slash (V,\left\|\cdot \right\|_V)$-measurable.
\end{lemma}
\begin{proof}[Proof of Lemma~\ref{lem:reg:w}]
Throughout this proof let 
$ \epsilon \colon L(W,V) \to V $ 
and 
$ g \colon (0,T)$ $\to V $ 
be the functions 
which satisfy for all $A \in L(W,V)$ that 
\begin{equation}
  \epsilon(A)=Aw 
  \qquad 
  \text{and}
  \qquad
  g= \epsilon \circ S
  . 
\end{equation}
Note that for all $a,b \in \R$, $A,B \in L(W,V)$ it holds that
\begin{align}
\epsilon (aA+bB)=a\, Aw +b \, Bw = a \,  \epsilon (A) + b  \, \epsilon (B)
\end{align}
and 
\begin{align}
\|\epsilon (A) \|_V= \|Aw\|_V \leq \|A\|_{L(W,V)} \|w\|_W.
\end{align}
Therefore, we obtain that 
\begin{equation}
  \epsilon \in L(L(W,V),V).
\end{equation}
This, in particular, proves that $\epsilon$ is $\mathcal{B}(L(W,V)) \slash \mathcal{B}(V)$-measurable.
This and the assumption $S$ is $\mathcal{B}((0,T)) \slash \mathcal{B}(L(W,V))$-measurable 
establish that  $ g= \epsilon \circ S$ is $\mathcal{B}((0,T)) \slash \mathcal{B}(V)$- measurable.
The fact that $((0,T) \ni s \mapsto (T-s) \in (0,T))$ is a $\mathcal{B}((0,T)) \slash \mathcal{B}((0,T))$ measurable function hence ensures that 
\begin{align}
((0,T) \ni s \mapsto S_{T-s} \, w \in V)
\end{align}
is a $\mathcal{B}((0,T)) \slash \mathcal{B}(V)$-measurable function.
This together with the assumption that  $(V, \left\|\cdot\right\|_V)$ is separable demonstrates that 
 the function 
\begin{equation}
(0,T) \ni s \mapsto S_{T-s} \, w \in V
\end{equation}
is strongly $\mathcal{B}((0,T)) \slash (V,\left\|\cdot\right\|_V)$-measurable.
The proof of Lemma~\ref{lem:reg:w} is thus completed.
\end{proof}

\begin{lemma}
\label{lem:regularity}
Let $(V, \left\|\cdot\right\|_V)$ 
be a separable $\R$-Banach space,
let $(W,\left\|\cdot\right\|_W)$ 
be an $\R$-Banach space, 
let $T \in (0, \infty)$,   $y \in C([0,T],W)$, 
and let $S \colon (0,T) \to L(W,V)$ be a $\mathcal{B}((0,T)) \slash \mathcal{B}(L(W,V))$-measurable function which satisfies that
\begin{equation}\label{eq:regularity:hp}\inf_{\alpha \in (0,1)}   \sup_{t \in (0,T)} t^{\alpha} \|S_t\|_{L(W,V)} < \infty.\end{equation}
Then   
\begin{enumerate}[(i)]
	\item \label{item:reg:y} it holds that the function $(0,T) \ni s \mapsto S_{T-s} \, y_s \in V$ is strongly $\mathcal{B}((0,T)) \slash (V,\left\|\cdot\right\|_V)$-measurable  and 
	\item \label{item:y:int} it holds that  $ \int_0^T \|S_{T-s} \, y_s \|_V \, ds < \infty$.
\end{enumerate}    
\end{lemma}
\begin{proof}[Proof of Lemma~\ref{lem:regularity}]
Throughout this proof
let $\alpha\in(0,1)$ be a real number which satisfies that
\begin{align}
\label{well:alpha}
\sup_{t\in(0,T)} t^\alpha \|S_t\|_{L(W,V)}< \infty,
\end{align}
let $x \colon (0,T) \to V$ be the function which satisfies for all $s \in (0,T)$ that $x(s) = S_{T-s} \, y_s$,
let $ \lf \cdot \rf_h \colon \R \to \R$, $ h \in (0, \infty)$, be the functions which satisfy for all $h \in (0, \infty)$, $t \in \R$ that 
\begin{align}
\lf s \rf_h = \max( (-\infty, s] \cap \{0, h, -h, 2h, -2h, \ldots\} ),
\end{align}
let $x_n \colon (0,T) \to V $, $n \in \N$, be the functions which satisfy for all  $s \in (0,T)$, $n \in \N$ that
\begin{align}
x_n(s) = S_{T-s} \, y_{\lf s \rf_{\frac{T}{n}}},
\end{align}
and for every set   $A \subseteq \R $  let  $\mathbbm{1}_A \colon \R \to \{0,1\}$ be the function which satisfies for all $ a \in A$ that $\mathbbm{1}_A(a)=1$ and for all $b\in \R \backslash A$ that $\mathbbm{1}_A(b)=0$.
Observe that the assumption 
that $y \in C([0,T],W)$ proves that for all $s \in (0,T)$ it holds that
\begin{align}
\label{eq:point:limit}
  \limsup_{ n \to \infty } 
  \|x_n(s) - x(s)\|_V 
  \leq 
  \|S_{T-s}\|_{L(W,V)} 
  \left[
    \limsup_{ n \to \infty } 
    \|y_{\lf s \rf_{\frac{T}{n}}} - y_s\|_V 
  \right] 
  = 0.
\end{align}
Moreover, note that for all $h \in (0, \infty)$, $s \in (0,T)$ it holds that
\begin{align}
S_{T-s}\, y_{\lfloor s \rfloor_{h}}= \sum_{k=0}^{\lfloor \frac{T}{h} \rfloor_1} \mathbbm{1}_{ [k h, (k+1)h)}(s) \, S_{T-s} \, y_{k h}.
\end{align}
Lemma~\ref{lem:reg:w} and the fact that for every strongly $\mathcal{B}((0,T)) \slash (V,\left\|\cdot\right\|_V)$-measurable function $f \colon (0, T) \to V$ and every strongly $\mathcal{B}((0,T)) \slash (V,\left\|\cdot\right\|_V)$-measurable function $g \colon (0, T) \to V$ it holds that the function $\big( (0,T) \ni s \mapsto (f(s) + g(s)) \in V \big)$ is  strongly $\mathcal{B}((0,T)) \slash (V,\left\|\cdot\right\|_V)$-measurable 
therefore prove that the functions $x_n$, $n \in \N$, are strongly $\mathcal{B}((0,T)) \slash (V,\left\|\cdot\right\|_V)$-measurable. Combining this  with \eqref{eq:point:limit} and
the fact that for every sequence of strongly $\mathcal{B}((0,T)) \slash (V,\left\|\cdot\right\|_V)$-measurable functions $f_n \colon (0, T) \to V$, $n \in \N$, and every function $f \colon (0, T) \to V$ with $ \forall \, s \in (0, T) \colon \limsup_{ n \to \infty } \|f_n(s) - f(s) \|_V = 0$ it holds that the function $f$ is strongly $\mathcal{B}((0,T)) \slash (V,\left\|\cdot\right\|_V)$-measurable 
 implies that the function $x$ is also strongly $\mathcal{B}((0,T)) \slash (V,\left\|\cdot\right\|_V)$-measurable. This establishes item~\eqref{item:reg:y}.
Next observe that \eqref{well:alpha} proves that 
\begin{align}
\label{eq:well:def2}
\begin{split}
&\int_0^T \|S_{T-s}\|_{L(W,V)} \, d s \leq \big[\! \sup\nolimits_{s \in (0,T)} s^{\alpha} \|S_s\|_{L(W,V)} \big] \int_0^T (T-s)^{-\alpha} \, d s\\
& = \frac{T^{1-\alpha}}{(1-\alpha)} \big[\! \sup\nolimits_{s \in (0,T)} s^{\alpha} \|S_s\|_{L(W,V)} \big]< \infty.
\end{split}
\end{align}
This together with the assumption that $y\in C([0,T],W)$ ensures that
\begin{align}
\begin{split}
\label{eq:well:def}
\int_0^T \|S_{T-s} \, y_s \|_V \, d s &\leq   \int_0^T \|S_{T-s}\|_{L(W,V)} \|y_s\|_W \, d s \\
&\leq \big[\!\sup\nolimits_{s \in [0,T]} \|y_s\|_W\big] \int_0^T \|S_{T-s}\|_{L(W,V)} \, d s < \infty.
\end{split}
\end{align}
This establishes item~\eqref{item:y:int}. The proof of Lemma~\ref{lem:regularity} is thus completed.
\end{proof}

\section{Continuity properties of mild solutions}
\label{sec:continuity}

In this section we present in Corollary~\ref{cor:mild.continuity}  an elementary result about the continuity  of mild solutions of certain nonlinear evolution equations.

\begin{cor}
\label{cor:mild.continuity}
Let $(V, \left\|\cdot\right\|_V)$ 
be a separable $\R$-Banach space,
let $(W,\left\|\cdot\right\|_W)$ 
be an $\R$-Banach space, 
let $T \in (0, \infty)$, $y \in C([0,T],W)$,
and  let $S \colon (0,T) \to L(W,V)$ be a $\mathcal{B}((0,T)) \slash \mathcal{B}(L(W,V))$-measurable function which satisfies that
\begin{equation}
\inf_{\alpha \in (0,1)}   
\sup_{t \in (0,T)} t^{\alpha} \|S_t\|_{L(W,V)} < \infty.
\end{equation}
Then   
\begin{enumerate}[(i)]
	\item \label{item:cont:y} 
	it holds for all $t\in(0,T]$  that the function $(0,t) \ni s \mapsto S_{t-s} \, y_s \in V$ is strongly $\mathcal{B}((0,t)) \slash (V,\left\|\cdot\right\|_V)$-measurable, 
	\item \label{item:cont:int} 
	it holds for all $t\in(0,T]$  that 
	$ \int_0^t \|S_{t-s} \, y_s \|_V \, ds < \infty$,
	and 
	\item\label{item:cont}
	it holds that the function
	$
	  [0,T] \ni t \mapsto \int^t_0 S_{t-s} \, y_s \, ds \in V
	$
	is continuous.
\end{enumerate} 
\end{cor}
\begin{proof}[Proof of Corollary~\ref{cor:mild.continuity}]
Throughout this proof  for every  set $A \subseteq \R$ let $\mathbbm{1}_A \colon \R$ $\to \{0,1\}$ be the function which satisfies for all $ a \in A$ that $\mathbbm{1}_A(a)=1$ and for all $b\in \R \backslash A$ that $\mathbbm{1}_A(b)=0$.
Note that items~\eqref{item:reg:y}--\eqref{item:y:int} in Lemma~\ref{lem:regularity} establish items~\eqref{item:cont:y}--\eqref{item:cont:int}.
It thus remains to prove item~\eqref{item:cont}.
For this let $\alpha\in(0,1)$ be a real number which satisfies that 
\begin{equation}
  \sup_{t \in (0,T)} t^{\alpha} \|S_t\|_{L(W,V)} < \infty
\end{equation}
and let 
$
  x \colon [0,T] \to V 
$
be the function which satisfies for all 
$t\in[0,T]$ that 
\begin{equation}
  x_t = \int^t_0 S_{t-s} \, y_s \, ds
  = \int^t_0 S_s \, y_{t-s} \, ds.
\end{equation}
Observe that for all 
$t_1,t_2\in[0,T]$ with $t_1\leq t_2$ it holds that 
\begin{equation}
\begin{split}
&
  \|x_{t_1}-x_{t_2}\|_V
\\&=
  \left\|
  \int^{t_2}_0
  S_s \, y_{t_2-s}
  \, ds
  -
  \int^{t_1}_0
  S_s \, y_{t_1-s}
  \, ds
  \right\|_V
\\&\leq
  \left\|
  \int^{t_2}_{t_1}
  S_s \, y_{t_2-s}
  \, ds
  \right\|_V
  +
  \left\|
  \int^{t_1}_0
  S_s \, (y_{t_2-s} - y_{t_1-s})
  \, ds
  \right\|_V
\\&\leq
  \int^{t_2}_{t_1}
  \|S_s\|_{L(W,V)} \, \|y_{t_2-s}\|_W
  \, ds
\\&\quad+
  \int^{t_1}_0
  \|S_s\|_{L(W,V)} \, 
  \|y_{\max\{t_2-s,0\}}-y_{\max\{t_1-s,0\}}\|_W
  \, ds
  .
\end{split}
\end{equation}
This implies that for all 
$t_1,t_2\in[0,T]$ with $t_1\leq t_2$ it holds that 
\begin{align*}
&
\|x_{t_1}-x_{t_2}\|_V
\\&\leq
  \bigg[
  \sup_{s\in(0,T)}
  s^\alpha \|S_s\|_{L(W,V)}
  \bigg] \,
  \bigg[
  \sup_{s\in[0,T]}
  \|y_s\|_W
  \bigg] \,
  \bigg[
  \int^{t_2}_{t_1}
  s^{-\alpha} \, ds
  \bigg]
  \\&\quad+
  \bigg[
  \sup_{s\in(0,T)}
  s^\alpha \|S_s\|_{L(W,V)}
  \bigg] \,
  \bigg[
  \int^T_0 
  s^{-\alpha}
  \|y_{\max\{t_2-s,0\}}-y_{\max\{t_1-s,0\}}\|_W \,
  \mathbbm{1}_{ [0,t_1] }(s)
  \,ds
  \bigg]
  \\&=
  \bigg[
  \sup_{s\in(0,T)}
  s^\alpha \|S_s\|_{L(W,V)}
  \bigg] \,
  \Bigg(
  \bigg[
  \sup_{s\in[0,T]}
  \|y_s\|_W
  \bigg] \,
  \bigg[
  \int^T_0
  s^{-\alpha} \, \mathbbm{1}_{ [ t_1, t_2] }(s) \, ds
  \bigg]
  \\&\quad+
  \int^T_0 
  s^{-\alpha}
  \|y_{\max\{t_2-s,0\}}-y_{\max\{t_1-s,0\}}\|_W \,
  \mathbbm{1}_{ [0,t_1] }(s)
  \,ds
  \Bigg)
  . \numberthis
\end{align*}
Lebesgue's theorem of dominated convergence hence ensures that for all 
$ t \in [0,T] $
and all 
$(t^{(n)}_1,t^{(n)}_2)\in[0,t] \times [t,T]$, $n\in\N$,
with 
$
  \limsup_{n\to\infty}
  |t^{(n)}_1 - t^{(n)}_2 |
  =0
$
it holds that 
\begin{equation}
  \limsup_{n\to\infty}
  \|x_{t^{(n)}_1}-x_{t^{(n)}_2}\|_V
  =0
  .
\end{equation}
This establishes item~\eqref{item:cont}. The proof of Corollary~\ref{cor:mild.continuity} is thus completed.
\end{proof}

\section{A perturbation estimate}
\label{sec:perturbation}

In  this section we establish in Lemma~\ref{lemma:perturbation}  a suitable perturbation estimate for mild solutions of nonlinear evolution equations. Lemma~\ref{lemma:perturbation} is an appropriate extended version of Andersson et al.~\cite[Proposition~2.7]{Andersson2015} and 
Jentzen \& Kurniawan~\cite[Corollary~3.1]{JentzenKurniawan2015}.

\begin{lemma}
\label{lemma:perturbation}
Let
 $ \Gamma \colon (0, \infty) \to (0, \infty)$ be the function which satisfies for all $x \in (0, \infty)$  that $\Gamma(x)= \int_0^{\infty} t^{(x-1)} \, e^{-t} \, dt$, let
 $ \mathfrak{E}_r \colon [0, \infty) \to [0, \infty)$, $r \in (0, \infty),$ be the functions which satisfy for all $ r \in (0, \infty)$, $x \in [0, \infty)$  that 
 $
 \mathfrak{E}_r[x]= \sum_{n=0}^{\infty} \tfrac{(x \, \Gamma(r))^{n}}{\Gamma(nr+1)} 
 $,
let $(V, \left\|\cdot\right\|_V)$ 
be a separable $\R$-Banach space,
let $(W,\left\|\cdot\right\|_W)$ 
be an $\R$-Banach space, 
let $T \in (0, \infty)$, $F \in C(V, W)$,  
$ x^1, x^2 \in C([0,T],V) $, 
$ \alpha \in (0,1) $,
let
$ S \colon (0,T) \to L(W,V) $ be a $\mathcal{B}((0,T)) \slash \mathcal{B}(L(W,V)) $-measurable function, let 
$
  \Psi \colon [0,\infty) \to [0,\infty)
$ 
be the function which
satisfies for all $ r \in [0,\infty) $ 
that 
\begin{equation}
\Psi(r) = 
\sup\!\left( \!
\left\{
\tfrac{ \| F(v) - F(w) \|_W 
}{
	\| v - w \|_V
} 
\colon v, w \in V, v \neq w, 
\|v\|_V + \|w\|_V \leq r\right\} \cup \{0\}  
\right)\!,
\end{equation}
and assume that
$
  \sup_{ t \in (0,T) } t^{ \alpha } \| S_t \|_{ L(W,V) } < \infty 
$. 
Then 
\begin{enumerate}[(i)]

\item 
\label{item:well_i}
it holds for all 
$ k \in \{ 1, 2 \} $,
$ t \in (0,T] $ that the 
function 
$
  (0,t) \ni s \mapsto S_{ t - s } \, F( x_s^k ) \in V 
$
is strongly 
$ \mathcal{B}( (0,t) ) \slash (V,\left\|\cdot\right\|_V)$-measurable,

\item 
\label{item:well_ii}
it holds for all 
$ k \in \{ 1, 2 \} $,
$ t \in (0,T] $ that 
$ \int_0^t \| S_{ t - s } \, F( x_s^k ) \|_V \, ds < \infty $, 
and

\item
\label{item:perturbation_iii}
it holds that 
\begin{multline}
\label{eq:perturb}
  \sup_{t \in [0,T]} \|x_t^1-x_t^2\|_V 
\\
 \leq 
  \mathfrak{E}_{ ( 1 - \alpha ) }\!\left[
    T^{ 1 - \alpha }
    \left(
      \sup_{ t \in (0,T) } t^{\alpha} \| S_t \|_{L(W,V)} 
    \right)
    \Psi\!\left(
      \sup_{ t \in [0,T]} 
      \left[ \|x_t^1\|_V + \|x_t^2\|_V \right]
    \right) 
  \right] 
\\
  \cdot \sup_{t \in [0,T]} \left\|x_t^1 -\smallint_0^t S_{t-s} \, F(x_s^1) \, ds  +  \smallint_0^t S_{t-s} \, F(x_s^2) \, ds- x_t^2\right\|_V < \infty.
\end{multline}
\end{enumerate}
\end{lemma}
\begin{proof}[Proof of Lemma~\ref{lemma:perturbation}]
First, note that 
Lemma~\ref{lem:regularity} establishes
items~\eqref{item:well_i}--\eqref{item:well_ii}.
It thus remains to prove item~\eqref{item:perturbation_iii}.
For this observe that the triangle inequality ensures that for all $ t \in [0,T] $ it holds that 
\begin{align}
\label{eq:per:triangle}
\begin{split}
  \|x_t^1 -x_t^2\|_V
&
  \leq 
  \left\|
    x_t^1 - \smallint_0^t S_{ t - s } \, F(x_s^1) \, ds 
    +  
    \smallint_0^t S_{ t - s } \, F(x_s^2) \, ds - x_t^2
  \right\|_V 
\\
& 
  + 
  \left\| 
    \smallint_0^t 
      S_{ t - s } \, \big[ F( x_s^1 ) - F( x_s^2 ) \big] 
    \, ds 
  \right\|_V\!.
\end{split}
\end{align}
Next note that for all $ t \in [0,T] $ it holds that
\begin{align}
\begin{split}
& \left\| 
    \smallint_0^t S_{t-s} \, \big[ F(x_s^1) - F(x_s^2) \big] \, ds \right\|_V \\
&\leq 
  \int_0^t \| S_{ t - s } \|_{ L(W, V ) } 
  \| F( x_s^1 ) - F( x_s^2 ) \|_W \, ds
\\
& \leq 
  \Psi\!\left(
    \sup_{ s \in [0,T] } 
    \left[ 
      \| x_s^1 \|_V + \| x_s^2 \|_V
    \right]
  \right) 
  \left[  
    \sup_{ s \in (0,T) } s^{ \alpha } \| S_s \|_{ L(W,V) } 
  \right] \\
  & \cdot
  \int_0^t 
    \left( t - s \right)^{ - \alpha } 
    \| x_s^1 - x_s^2 \|_V \, 
  ds
  .
\end{split}
\end{align}
Combining this with \eqref{eq:per:triangle} proves that for all $t \in [0,T]$ it holds that
\begin{align}
\begin{split}
\|x_t^1 -x_t^2\|_V &\leq \left\|x_t^1 -\smallint_0^t S_{t-s} \, F(x_s^1) \, ds  +  \smallint_0^t S_{t-s} \, F(x_s^2) \, ds- x_t^2\right\|_V \\
&
  +
  \Psi\!\left(
    \sup_{ s \in [0,T] } 
    \left[ 
      \| x_s^1 \|_V + \| x_s^2 \|_V
    \right]
  \right) 
  \left[  
    \sup_{ s \in (0,T) } s^{ \alpha } \| S_s \|_{ L(W,V) } 
  \right] \\
 & \cdot \int_0^t 
    \left( t - s \right)^{ - \alpha } 
    \| x_s^1 - x_s^2 \|_V \, 
  ds
  .
\end{split}
\end{align}
The generalized Gronwall inequality (see, e.g.,  Henry~\cite[Lemma~7.1.1]{h81}) and Corollary~\ref{cor:mild.continuity} hence establish \eqref{eq:perturb}.
The proof of Lemma~\ref{lemma:perturbation} is thus completed.
\end{proof}

\section{Uniqueness of mild solutions}
\label{sec:unique}

In this section  we present in Corollary~\ref{cor:local:unique} an elementary fact about uniqueness of mild solutions of certain nonlinear evolution equations. Corollary~\ref{cor:local:unique}
is an immediate consequence of  Lemma~\ref{lemma:perturbation} in Section~\ref{sec:perturbation} above.

\begin{cor}
\label{cor:local:unique}
Let $(V, \left\|\cdot\right\|_V)$ 
be a separable $\R$-Banach space,
let $(W,\left\|\cdot\right\|_W)$ 
be an $\R$-Banach space, 
let 
$ T \in (0, \infty) $, $ \tau \in (0,T] $, $ F \in C(V, W) $,  
$ x^1, x^2, o \in C( [0,\tau], V ) $,
and let $ S \colon (0,T) \to L(W,V) $ be a $\mathcal{B}((0,T)) \slash \mathcal{B}(L(W,V)) $-measurable function which
satisfies
for all $ r \in [0,\infty) $, $t \in [0,\tau]$, $k \in \{1,2\}$ that
\begin{multline}
\sup \!\left( \! \left\{\tfrac{\|F(v)-F(w)\|_W}{\|v-w\|_V} \colon v, w \in V, v \neq w, \|v\|_V + \|w\|_V \leq r\right\} \cup \{0\}  \right) \\
+ \left[\inf_{\alpha \in (0,1)}\sup_{s \in (0,T)} s^{\alpha} \|S_s\|_{L(W,V)} \right]< \infty
\end{multline}
and 
\begin{align}
  x_t^k = \int_0^t S_{t-s} \, F(x_s^k) \, ds + o_t.
\end{align}
Then it holds that $ x^1 = x^2 $.
\end{cor}

\section{Local existence of mild solutions}
\label{sec:local}

In this section we present in Theorem~\ref{thm:existence} a local and existence and uniqueness result for certain nonlinear evolution equations. Our proof of Theorem~\ref{thm:existence} is based on well-known techniques from the literature on evolution equations and uses the Banach fixed-point theorem, Corollary~\ref{cor:local:unique} from Section~\ref{sec:unique}, Corollary~\ref{cor:complete:R} from Section~\ref{sec:complete}, and Corollary~\ref{cor:mild.continuity} from Section~\ref{sec:continuity}.

\begin{theorem}
	\label{thm:existence}
	Let $(V, \left\|\cdot\right\|_V)$ 
	be a separable $\R$-Banach space,
	let $(W,\left\|\cdot\right\|_W)$ 
	be an $\R$-Banach space, 
	let $T \in (0, \infty)$, $F \in C(V, W)$, $o \in C([0,T],V)$,
	and let $ S \colon (0,T) \to L(W,V) $ be a $\mathcal{B}((0,T)) \slash \mathcal{B}(L(W,V)) $-measurable function which
	satisfies for all $r \in [0,\infty)$ that 
	\begin{multline}
	\sup \!\left( \! \left\{\tfrac{\|F(v)-F(w)\|_W}{\|v-w\|_V} \colon v, w \in V, v \neq w, \|v\|_V + \|w\|_V \leq r\right\} \cup \{0\}  \right) \\
	+ \left[ \inf_{\alpha \in (0,1)}\sup_{s \in (0,T)} s^{\alpha} \|S_s\|_{L(W,V)} \right] < \infty.
	\end{multline}
	Then there exists a real number $\tau \in (0,T]$ such that 
	there exists a unique  $x \in C( [0, \tau],  V)$ which satisfies that 
	\begin{enumerate}[(i)]
		\item
		\label{item:strongly} 
		it holds for all $ t \in (0,\tau] $ that the 
		function 
		$
		(0,t) \ni s \mapsto S_{ t - s } \, F( x_s ) \in V 
		$
		is strongly 
		$ \mathcal{B}( (0,t) ) \slash (V,\left\|\cdot\right\|_V)$-measurable,
		\item 
		\label{item:finite}
		it holds for all $ t \in [0, \tau] $ that 
		$ \int_0^t \| S_{t-s} \, F(x_s) \|_V \, ds < \infty $, 
		and
		\item
		\label{item:solution}
		it holds for all $ t \in [0, \tau] $ that 
		\begin{align}
		\label{eq:existence}
		x_t = \int_0^t S_{t-s} \, F(x_s) \, ds + o_t.
		\end{align}
	\end{enumerate}
\end{theorem} 
\begin{proof}[Proof of Theorem~\ref{thm:existence}]
	Throughout this proof let $\Psi \colon [0,\infty) \to [0,\infty)$ be the function which satisfies for all $r \in [0,\infty)$ that
	\begin{align}
	\begin{split}
	&\Psi(r) \\
	&= \sup \left( \!\left\{\tfrac{\|F(v)-F(w)\|_W}{\|v-w\|_V} \colon v, w \in V, v \neq w, \|v\|_V + \|w\|_V \leq r\right\} \cup \{0\}  \right)\!,
	\end{split}
	\end{align}
	let $\alpha \in (0,1)$ be a real number which satisfies that
	\begin{equation}
	\sup_{t \in (0,T)} t^{\alpha} \|S_t\|_{L(W,V)} < \infty,
	\end{equation}
	let $R \in [0,\infty)$ be the real number given by $R= \sup_{s \in [0,T]} \|o_s\|_V$, let $\Xi_{\tau} \subseteq C([0,\tau],V), \tau \in (0,T]$, be the sets which satisfy for all $\tau \in (0,T]$ that 
	\begin{align}
	\begin{split}
	\Xi_{\tau} = \left\{x \in C([0,\tau],V): \sup\nolimits_{t \in [0,\tau]} \|x_t\|_V \leq R +1 \right\}\!, 
	\end{split}
	\end{align}
	and let $\Phi_{\tau} \colon C([0,\tau], V) \to C([0,\tau], V)$, $\tau \in (0,T]$, be the functions which satisfy for all $\tau \in (0,T]$, $x \in C([0,\tau], V)$,  $t \in [0,\tau]$ that
	\begin{align}
	(\Phi_{\tau} (x))(t) = \int_0^{ t } S_{t-s} \, F(x_s) \, ds + o_t.
	\end{align}
	Observe that Corollary~\ref{cor:mild.continuity} and the assumption that $o \in C([0,T],V)$ ensure that the functions $\Phi_{\tau}$, $\tau \in (0,T]$, 
	are well-defined.
	Moreover, observe that for all $\tau \in (0,T]$, $t \in [0, \tau]$ it holds that
	\begin{align}
	\begin{split}
	&
	\|(\Phi_{\tau} (0))(t)\|_V 
	\\ 
	&\leq \|o_t\|_V + \int_0^t \|S_{t-s} \, F(0)\|_V \, ds \leq \|o_t\|_V + \int_0^t \|S_{t-s}\|_{L(W,V)} \|F(0)\|_W \, ds\\
	& \leq  \|o_t\|_V + \big[\! \sup\nolimits_{s \in (0,T)} s^{\alpha} \|S_s\|_{L(W,V)} \big] \|F(0)\|_W \int_0^t (t-s)^{-\alpha} \, ds\\
	& \leq R + \frac{t^{(1-\alpha)}}{(1-\alpha)} \big[\! \sup\nolimits_{s \in (0,T)} s^{\alpha} \|S_s\|_{L(W,V)} \big] \|F(0)\|_W < \infty.
	\end{split}
	\end{align}
	Hence, we obtain that for all $\tau \in (0,T]$ it holds that 
	\begin{align}
	\label{eq:0:bound}
	\|\Phi_{\tau} (0)\|_{C([0,\tau],V)}\leq R + \frac{{\tau}^{(1-\alpha)}}{(1-\alpha)} \big[\! \sup\nolimits_{s \in (0,T)} s^{\alpha} \|S_s\|_{L(W,V)} \big] \|F(0)\|_W < \infty.
	\end{align}
	Next note that for all $\tau \in (0,T]$, $t \in [0, \tau]$, $x, y \in \Xi_{\tau}$ it holds that
	\begin{align}
	\begin{split}
	&\|(\Phi_{\tau}(x))(t) - (\Phi_{\tau}(y))(t)\|_V = \left\| \int_0^t S_{t-s}[F(x_s)-F(y_s)] \, ds \right\|_V\\
	& \leq \int_0^{t} \|S_{t-s}\|_{L(W,V)} \|F(x_s)- F(y_s)\|_V \, ds\\
	& 
	\leq \Psi(\sup\nolimits_{s \in [0,\tau]} [\|x_s\|_V + \|y_s\|_V]) 
	\big[\! \sup\nolimits_{s \in (0,T)} s^{\alpha} \|S_s\|_{L(W,V)} \big] 
	\\
	& \quad
	\cdot \int_0^{t} (t-s)^{-\alpha} \|x_s-  y_s\|_V \, ds\\
	& \leq \frac{t^{(1-\alpha)}}{(1-\alpha)} \Psi(2R+2) \big[\! \sup\nolimits_{s \in (0,T)} s^{\alpha} \|S_s\|_{L(W,V)} \big] \|x-y\|_{C([0,\tau],V)}<\infty.
	\end{split}
	\end{align}
	Therefore, we obtain that for all $\tau \in (0,T]$, $x, y \in \Xi_{\tau}$ 
	it holds that
	\begin{align}
	\label{eq:x-y:bound}
	\begin{split}
	&\|\Phi_{\tau}(x) - \Phi_{\tau}(y)\|_{C([0,\tau],V)} \\
	&\leq \frac{\tau^{(1-\alpha)}}{(1-\alpha)} \Psi(2R+2) \big[\! \sup\nolimits_{s \in (0,T)} s^{\alpha} \|S_s\|_{L(W,V)} \big]\|x-y\|_{C([0,\tau],V)} < \infty.
	\end{split}
	\end{align}
	Combining this and \eqref{eq:0:bound} ensures that for all $\tau \in (0,T]$, $ x \in \Xi_{\tau}$ it holds that
	\begin{align*}
	&  \|\Phi_{\tau}(x) \|_{C([0,\tau],V)} \leq \|\Phi_{\tau}(x) - \Phi_{\tau}(0)\|_{C([0,\tau],V)}  + \|\Phi_{\tau} (0)\|_{C([0,\tau],V)}
	\\
	&\leq  \frac{\tau^{(1-\alpha)}}{(1-\alpha)} \Psi(2R+2) \big[\! \sup\nolimits_{s \in (0,T)} s^{\alpha} \|S_s\|_{L(W,V)} \big]\|x\|_{C([0,\tau],V)} \numberthis
	\\
	&
	\quad
	+ R + \frac{{\tau}^{(1-\alpha)}}{(1-\alpha)} \big[\! \sup\nolimits_{s \in (0,T)} s^{\alpha} \|S_s\|_{L(W,V)} \big] \|F(0)\|_W
	\\
	& \leq R + \frac{\tau^{(1-\alpha)}}{(1-\alpha)} 
	\big[\! \sup\nolimits_{s \in (0,T)} s^{\alpha} \|S_s\|_{L(W,V)} \big]  \big[\Psi(2R+2) ( R + 1 ) +  \|F(0)\|_W \big] <\infty.
	\end{align*}
	The fact that $\alpha < 1$ and \eqref{eq:x-y:bound} therefore imply that there exists a real number $\tau \in (0,T]$ such that for all $x, y \in \Xi_{\tau}$ it holds that
	\begin{align}
	\begin{split}
	\|\Phi_{\tau}(x) \|_{C([0,\tau],V)} &\leq  R +1 
	\end{split}
	\end{align}
	and
	\begin{align}
	\label{eq:contraction}
	\|\Phi_{\tau}(x) - \Phi_{\tau}(y)\|_{C([0,\tau],V)} &\leq \tfrac{1}{2} \|x-y\|_{C([0,\tau],V)}.
	\end{align}
	This ensures that $\Phi_{\tau}(\Xi_{\tau}) \subseteq \Xi_{\tau}$. 
	The Banach fixed-point theorem, Corollary~\ref{cor:complete:R}, 
	and \eqref{eq:contraction} hence demonstrate that 
	there exists a unique function $x \in \Xi_{\tau}$  such that 
	\begin{equation}
	\Phi_{\tau}(x)=x
	.
	\end{equation}
	Combining this and Corollary~\ref{cor:local:unique} 
	establishes items~\eqref{item:strongly}--\eqref{item:solution}.
The proof of Theorem~\ref{thm:existence} is thus completed.
\end{proof}

\section{Local existence of maximal mild solutions}
\label{sec:main_result}

In this section we employ well-known techniques from the literature on evolution equations to establish in Theorem~\ref{thm:unique} and Corollary~\ref{cor:unique} the unique local existence of maximal mild solutions of the considered nonlinear evolution equations.

\begin{lemma}
\label{lemma:extension}
Let $(V, \left\|\cdot\right\|_V)$ 
be a separable $\mathbb{R}$-Banach space,
let $(W,\left\|\cdot\right\|_W)$ 
be an $\mathbb{R}$-Banach space, 
let $T \in (0, \infty)$, 
$\tau\in (0,T]$, 
$x \in C([0,\tau),V)$,
$o \in C([0,T],V)$, 
$F \in C(V, W)$,
assume
for all
$r \in [0,\infty)$
that
$\limsup_{s \nearrow \tau} \|x_s\|_V < \infty$ and
\begin{equation}
\sup \! \left(\! \left\{\tfrac{\|F(v)-F(w)\|_W}{\|v-w\|_V} \colon v, w \in V, v \neq w, \|v\|_V + \|w\|_V \leq r\right\} \cup \{0\}  \right) <\infty,
\end{equation}  
let $S \colon (0,T) \to L(W,V)$ be a $\mathcal{B}((0,T))\slash \mathcal{B}(L(W,V))$-measurable function, 
let $\phi \colon (0,T) \to (0,\infty)$ be a  function which satisfies
 that 
$ \limsup_{s \searrow 0} \phi(s)=0$
and
\begin{equation} 
\inf_{\alpha \in (0,1)} \sup_{s \in (0,T)} \left[s^{\alpha} \|S_s\|_{L(W,V)} 
+
s^{\alpha}  \sup_{u \in (s,T)}\! \left( \tfrac{\|S_{u} - S_{s} \|_{L(W,V)} }{ |\phi(u-s)| }	\right)\right] \!< \infty,
\end{equation}
and assume 
for all  $t \in [0,\tau)$ that  
$ \int_0^t \|S_{t-s} \, F(x_s) \|_V \, ds < \infty$
and
\begin{align}
\label{eq:unique}
x_t = \int_0^t S_{t-s} \, F(x_s) \, ds + o_t.
\end{align}
Then 
\begin{enumerate}[(i)]
	\item \label{item:extension:1}
there exists a unique  $y \in C( [0,\tau], V)$ which satisfies that
$y|_{[0, \tau)} = x$ and
	\item \label{item:extension:2}
it holds for all $t\in [0,\tau]$ that 
$ \int_0^t \|S_{t-s} \, F(y_s) \|_V \, ds < \infty$ and
\begin{equation}
y_t = \int_0^t S_{t-s} \, F(y_s) \, ds + o_t.
\end{equation}
\end{enumerate}
\end{lemma}
\begin{proof}[Proof of Lemma~\ref{lemma:extension}]
First, note that the assumption that $\limsup_{s \nearrow \tau} \|x_s\|_V < \infty$ and the assumption that $x\in C([0,\tau),V)$ ensure that
\begin{equation}
\sup_{s\in [0,\tau)} \|x_s\|_V<\infty.
\end{equation}	
Next 
let $\alpha\in (0,1)$ be a real number which satisfies that
\begin{align} \label{eq:extension:ass}
\sup_{s \in (0,T)} 
\left[ s^{\alpha} \|S_s\|_{L(W,V)} 
+ s^{\alpha}  \sup_{t \in (s,T)}\! \left( \tfrac{\|S_{t} - S_{s} \|_{L(W,V)} }{ |\phi(t-s)| }	\right)\right] \! <\infty,
\end{align}
let $R\in (0,\infty)$ be  the  real number given by
\begin{align} \label{eq:extension:R}
\begin{split}
R &= \left[ \sup_{s \in (0,T)} \left[ s^{\alpha} \|S_s\|_{L(W,V)} 
+ s^{\alpha}  \sup_{t \in (s,T)}\! \left( \tfrac{\|S_{t} - S_{s} \|_{L(W,V)} }{ |\phi(t-s)| }	\right)\right] \right] \\
& \quad + \left[\sup_{s\in [0,\tau)} \|x_s\|_V \right]  + \|F(0)\|_W ,
\end{split}
\end{align}
let $\Psi \colon [0,\infty) \to [0,\infty)$ be the function which satisfies for all $r \in [0,\infty)$ that
\begin{align}
\begin{split}
& \Psi(r) = \sup \left( \left\{\tfrac{\|F(v)-F(w)\|_W}{\|v-w\|_V} \colon v, w \in V, v \neq w, \|v\|_V + \|w\|_V \leq r\right\} \cup \{0\}  \right) \! ,
\end{split}
\end{align}
let $\delta_{\varepsilon} \in (0,\infty)$, $\varepsilon \in (0, \infty)$, be the real numbers which satisfy for all $\varepsilon \in (0, \infty)$ that
\begin{equation}
\label{eq:delta}
	\delta_{\varepsilon} 
	= 
	\min\!\left\{ 
		\tau,   \tfrac{(1-\alpha) \varepsilon }{3 R^2  \left(1
+\Psi(R)\right) \tau^{1-\alpha}}
	\right\} \! ,
\end{equation}
and let $\gamma_{\varepsilon} \in (0, \tau |\delta_{\varepsilon}|^{\nicefrac{1}{(1-\alpha)}}]$, $\varepsilon \in (0, \infty)$, be real numbers which satisfy for all $\varepsilon \in (0, \infty)$ that
\begin{equation}
\label{eq:gamma}
\sup_{s \in (0, \gamma_{\varepsilon})} \phi(s) \leq \delta_{\varepsilon}
\end{equation}
and
\begin{equation}
\label{eq:cont:o}
\sup_{s \in (0, \gamma_{\varepsilon})} \|o(\tau-s) - o_{\tau} \|_V \leq \tfrac{\varepsilon}{6} .
\end{equation}
Note that the triangle inequality implies that for all $t_1, t_2 \in [0,\tau)$ with $t_1 \leq t_2$ it holds that
\begin{align}
\label{eq:xCauchy}
\begin{split}
 & \|x_{t_2}- x_{t_1}\|_V 
 \\
 	& \leq
 \|o_{t_2} -o_{t_1}\|_V 
 + \left\|\int_0^{t_1} S_{t_2-s} F(x_s) - S_{t_1-s} F(x_s) \, ds \right\|_V
 	\\
	& \quad
+ \left\|\int_{t_1}^{t_2} S_{t_2-s} F(x_s) \, ds \right\|_V 
		\\
		& 
 \leq 
 \|o_{t_2} -o_{t_1}\|_V
 	\\
 	& \quad 
 + \int_0^{t_1} \| S_{t_2-s} - S_{t_1-s}\|_{L(W,V)} \big(\|F(x_s)-F(0)\|_W +\|F(0)\|_W \big) \, ds
 	\\
 	& \quad 
 + \int_{t_1}^{t_2} \| S_{t_2-s} \|_{L(W,V)} \big(\|F(x_s)-F(0)\|_W +\|F(0)\|_W \big) \, d s
.
\end{split}
\end{align}
Furthermore, observe that \eqref{eq:extension:R} and the fact that the function $\Psi \colon [0,\infty) \to [0,\infty)$ is non-decreasing demonstrate that for all
$s \in [0, \tau)$ it holds that
\begin{equation}
\begin{split}
&\|F(x_s)-F(0)\|_W +\|F(0)\|_W \\
&= \tfrac{\|F(x_s)-F(0)\|_W}{\|x_s\|_V} \big[ \|x_s\|_V \big] +\|F(0)\|_W\\
& \leq \Psi\big(\|x_s\|_V\big) \|x_s\|_V + \|F(0)\|_W\\
& \leq \Psi\bigg(\sup_{t\in [0,\tau)} \|x_t\|_V\! \bigg) \bigg[\sup_{t\in [0,\tau)} \|x_t\|_V\bigg] + \|F(0)\|_W\\
& \leq R \big( \Psi(R) + 1\big).
\end{split}
\end{equation}
This and \eqref{eq:extension:R} ensure that
for all $t_1, t_2 \in [0,\tau)$ with $t_1 \leq t_2$ it holds that
\begin{align*}  
\label{eq:secondterm}
	& 
\int_0^{t_1} \| S_{t_2-s} - S_{t_1-s}\|_{L(W,V)} \big(\|F(x_s)-F(0)\|_W +\|F(0)\|_W \big)  \, ds
	\\
&	\leq  R \big( \Psi(R) + 1\big) \phi(t_2-t_1) \int_0^{t_1} \left[ \tfrac{\| S_{t_2-s} - S_{t_1-s}\|_{L(W,V)}(t_1-s)^{\alpha}}{\phi(t_2-t_1)} \right] (t_1- s)^{-\alpha} \, ds\\ 
	& 
\leq  R \big( \Psi(R) + 1\big) \phi(t_2-t_1) 
\left[ \sup_{ \substack{u \in (0,T),\\ v \in (u,T)}} \! \left(  \tfrac{u^{\alpha} \|S_{v} - S_{u} \|_{L(W,V)}}{|\phi(v-u)|} \right)\right] 
 \int_0^{t_1} (t_1-s)^{-\alpha} \, ds
	\\
	& \leq   R^2 \big( \Psi(R) + 1 \big) \phi(t_2-t_1) \left[ \frac{\tau^{(1-\alpha)}}{(1-\alpha)} \right] \numberthis
\end{align*}
and 
\begin{align} \label{eq:thirdterm}
\begin{split}
& \int_{t_1}^{t_2} \| S_{t_2-s} \|_{L(W,V)} \left(\|F(x_s)-F(0)\|_W +\|F(0)\|_W \right) d s
	\\
	& \leq  R \big( \Psi(R) + 1\big) 
 \left[\sup_{u \in (0,T)} \left( u^{\alpha} \|S_u\|_{L(W,V)} \right) \right]  
 \int_{t_1}^{t_2} (t_2-s)^{-\alpha} \, ds
	\\
	&  
\leq R^2 \big(\Psi(R) +1\big) \left[ \frac{(t_2-t_1)^{(1-\alpha)}}{(1-\alpha)} \right].
\end{split}
\end{align}
Combining this with 
\eqref{eq:xCauchy} establishes that for all $t_1,t_2\in [0,\tau)$ it holds that
\begin{align}
\label{eq:xCauchy:new}
\begin{split}
 & \|x_{t_2}- x_{t_1}\|_V 
 		\\ 
 		& \leq
\|o_{t_2} -o_{t_1}\|_V 
+ R^2 \big(\Psi(R) +1 \big) \left(\frac{ |t_2-t_1|^{(1-\alpha)} }{(1-\alpha)} 
+ \frac{\phi(|t_2-t_1|) \,  \tau^{(1-\alpha)} }{(1-\alpha)} \right) \! .
\end{split}
\end{align}
In addition, observe that \eqref{eq:cont:o} and the triangle inequality imply that for all $\varepsilon \in (0, \infty)$, $t_1,t_2 \in (\tau-\gamma_{\varepsilon},\tau)$ it holds that
\begin{align} \label{eq:o}
\|o_{t_1}-o_{t_2}\|_V \leq \|o_{t_1} - o_{\tau}\|_V + \|o_{t_2} - o_{\tau} \|_V \leq \tfrac{\varepsilon}{3}.
\end{align}
This, \eqref{eq:delta}, and \eqref{eq:gamma} demonstrate that for all $\varepsilon \in (0, \infty)$, 
$t_1,t_2 \in (\tau-\gamma_{\varepsilon},\tau)$ it holds that
\begin{equation}
\begin{split}
 & \|x_{t_2}- x_{t_1}\|_V 
\\ 
& \leq \frac{\varepsilon}{3} + R^2 \big(\Psi(R) +1 \big) \left(\frac{ |\gamma_{\varepsilon}|^{(1-\alpha)} }{(1-\alpha)} 
+ \frac{\delta_{\varepsilon} \,  \tau^{(1-\alpha)} }{(1-\alpha)} \right)\\
& \leq \frac{\varepsilon}{3} + R^2 \big(\Psi(R) +1 \big) \left[ \frac{2\delta_{\varepsilon} \,  \tau^{(1-\alpha)} }{(1-\alpha)} \right] \leq \frac{\varepsilon}{3} + \frac{2\varepsilon}{3} = \varepsilon.
\end{split}
\end{equation}
Combining this 
with the fact that 
$
\sup_{s\in [0,\tau)} \|x_s\|_V \leq R
$ 
ensures that for every sequence $t_n \in [0,\tau)$, $n \in \N$, with $\limsup_{n\to\infty}|t_n-\tau|=0$ it holds  
\begin{enumerate}[(a)]
	\item that  $x_{t_n} \in V$, $n \in \N$, is a Cauchy sequence and
	\item that $\{x_{t_n} \in V \colon n\in \N \}\subseteq \{v \in V \colon \|v\|_V \leq R\}$.
\end{enumerate}
The fact that $\{v \in V \colon \|v\|_V \leq R\}$ is a closed subset of the $\R$-Banach space $(V,\left\|\cdot\right\|_V)$ hence implies that there exists a unique $w \in \{v \in V \colon \|v\|_V \leq R\}$ which satisfies  that
\begin{equation}
\label{eq:exist:w}
\limsup_{s\nearrow \tau} \|x_s-w\|_V=0.
\end{equation}
Next let $y \colon [0,\tau] \to V$ be 
 the function  which satisfies for all $t\in [0,\tau]$ that 
\begin{equation}
y_t =
\begin{cases}
x_t  &\colon   t \in [0,\tau)\\
w  &\colon   t=\tau
\end{cases}.
\end{equation}
Observe that \eqref{eq:exist:w} ensures that  $y \in C([0, \tau], V)$. This establishes item~\eqref{item:extension:1}.
It thus remains to prove item~\eqref{item:extension:2}.
For this observe that
Corollary~\ref{cor:mild.continuity} 
(with 
$T=\tau$, $\mathcal{S}=\mathcal{S}|_{(0,\tau)}$, and $y= \left([0,\tau] \ni t \mapsto F(y_t)\in V\right)$ in the notation of Corollary~\ref{cor:mild.continuity}) 
demonstrates that for all $t\in [0,\tau]$ it holds that $ \int_0^t \|S_{t-s} \, F(y_s) \|_V \, ds < \infty$ and 
\begin{equation}
y_t = \int_0^t S_{t-s} \, F(y_s) \, ds + o_t.
\end{equation}
This establishes item~\eqref{item:extension:2}.
The proof of Lemma~\ref{lemma:extension} is thus completed.
\end{proof}

\begin{lemma}
\label{lemma:open}
Let $(V, \left\|\cdot\right\|_V)$ 
be a separable $\mathbb{R}$-Banach space,
let $(W,\left\|\cdot\right\|_W)$ 
be an $\mathbb{R}$-Banach space, 
let $F \in C(V, W)$, $T \in (0, \infty)$, 
let  $J \subseteq [0,T]$ be a convex set which satisfies that  $0 \in J$ and $ 0<\sup(J)<T$, 
assume for all $r \in [0,\infty)$ that
\begin{equation}
\sup \left( \!\left\{\tfrac{\|F(v)-F(w)\|_W}{\|v-w\|_V} \colon v, w \in V, v \neq w, \|v\|_V + \|w\|_V \leq r\right\} \cup \{0\}  \right)<\infty,
\end{equation}
let $S \colon (0,T) \to L(W,V)$ be a $\mathcal{B}((0,T))\slash \mathcal{B}(L(W,V))$-measurable function which satisfies that
\begin{equation}
\inf_{\alpha \in (0,1)} \sup_{s \in (0,T)} s^{\alpha} \|S_s\|_{L(W,V)} < \infty,
\end{equation}
let $\mathcal{S} \colon [0,T] \to L(V)$ be a function
which satisfies for all 
 $t_1\in [0,T)$, $t_2 \in (0,T-t_1) $, $u \in V$ 
that $([0,T] \ni s \mapsto \mathcal{S}_s u \in V) \in C([0,T],V)$ and
$S_{t_1+t_2}=\mathcal{S}_{t_1} S_{t_2}$,   
let $o \in C([0,T],V)$, 
$x \in C(J,V)$ satisfy for all 
$t \in J$ 
that
 $ \int_0^t \|S_{t-s}$ $F(x_s) \|_V \, ds < \infty$ 
and
 \begin{align} \label{open:x}
 x_t = \int_0^t S_{t-s} \, F(x_s) \, ds + o_t,
 \end{align}
and assume for all convex sets $I \subseteq [0,T]$ with $I \supseteq J $ and all  $y \in C(I,  V)$ with 
$y|_{J}=x$, 
$\forall \, t \in I \colon \int_0^t \|S_{t-s} \, F(y_s) \|_V \, ds < \infty$, and 
$ \forall \, t \in I \colon  y_t = \int_0^t S_{t-s} \, F(y_s) \, ds $ $+ o_t$ that $I=J$. Then 
\begin{align} \label{open:halfopen}
J= [0,\sup(J)).
\end{align}
\end{lemma}
\begin{proof}[Proof of Lemma~\ref{lemma:open}]
Throughout this proof let $\tau\in [0,T)$ be the real number given by $\tau= \sup(J)$. We prove Lemma~\ref{lemma:open} by contradiction. We thus assume that $\tau \in J$. This ensures that 
\begin{align} \label{open:closed}
J=[0,\tau].
\end{align}
Next note that Theorem~\ref{thm:existence} 
(with $T=T-\tau$, 
 $F=F$, 
$o= ([0,T-\tau] \ni t \mapsto \mathcal{S}_t(x_{\tau}-o_{\tau}) + o_{\tau+t} \in V) \in C([0,T-\tau],V)$,
$S=S|_{(0,T-\tau)}$ in the notation of Theorem~\ref{thm:existence}) 
demonstrates that there exist a real number $\varepsilon \in (0,T-\tau]$ and a  function $y \in C( [0, \varepsilon],  V)$ such that
 \begin{enumerate}[(i)]
 	\item  for all $ t \in (0,\varepsilon] $ it holds that  the 
 	function 
 	$
 	(0,t) \ni s \mapsto S_{ t - s } F( y_s ) \in V 
 	$
 	is strongly 
 	$ \mathcal{B}( (0,t) ) \slash (V,\left\|\cdot\right\|_V)$-measurable,
 	\item  for all $ t \in [0, \varepsilon] $ it holds  that $ \int_0^t \| S_{t-s} \, F(y_s) \|_V \, ds < \infty $,
 	and
 	\item  for all $ t \in [0, \varepsilon] $ it holds  that
 		\begin{align}
 	\label{eq:open:existence}
 	y_t = \int_0^t S_{t-s} \, F(y_s) \, ds + \mathcal{S}_t(x_{\tau}-o_{\tau}) + o_{\tau+t}.
 	\end{align}
 \end{enumerate}
In the next step let $u\colon [0,\tau + \varepsilon] \to V$ be the function which satisfies for all $t\in [0,\tau + \varepsilon]$  that
		\begin{align} \label{eq:open:u}
		u_t=\begin{cases} x_t  &\colon   t\in J  \\
		y_{t-\tau} &\colon  t\in (\tau,\tau+\varepsilon]
		\end{cases}.
		\end{align}
Observe that  \eqref{open:x} and the hypothesis that $\tau\in J$  ensure that
\begin{equation}
\label{eq:x:tau}
x_\tau =\int_{0}^{\tau} S_{\tau-s} \, F(x_s) \, ds + o_\tau.
\end{equation}
Combining this and \eqref{eq:open:existence} implies that for all $t\in [0,\varepsilon]$ it holds that
\begin{align} \label{eq:open:y}
\begin{split}
y_t &= \int_{0}^{t} S_{t-s} \, F(y_s) \, ds + \mathcal{S}_t \!  \int_{0}^{\tau} S_{\tau-s} \, F(x_s) \, ds + o_{\tau+t}\\
& = \int_{0}^{t} S_{t-s} \, F(y_s) \, ds + \int_{0}^{\tau} S_{\tau+t-s} \, F(x_s) \, ds + o_{\tau+t}.
\end{split}
\end{align}
This and \eqref{eq:x:tau} assure that $y_0=x_\tau$.
Hence, we obtain that for all $t\in (\tau,\tau+\varepsilon]$ it holds that $u\in C([0,\tau+\varepsilon], V)$ and 
\begin{align} \label{eq:open:useful}
	\begin{split}
& \int_0^t \|S_{t-s} \, F(u_s) \|_V \, d s  
	\\
	& = 
\int_0^\tau \|S_{t-s} \, F(x_s) \|_V \, d s + \int_{\tau}^t \|S_{t-s} \, F(y_{s-\tau}) \|_V \, d s
	\\
	& = 
\int_0^\tau \|\mathcal{S}_{t-\tau} S_{\tau-s} \, F(x_s) \|_V \, d s + \int_0^{t-\tau} \! \|S_{t-\tau -s} \, F(y_{s}) \|_V \, d s
	\\
	& \leq 
\|\mathcal{S}_{t-\tau}\|_{L(V)} \! \int_0^\tau \|S_{\tau-s} \, F(x_s) \|_V \, d s + \int_0^{t-\tau} \! \|S_{t-\tau -s} \, F(y_{s}) \|_V \, d s
 < \infty.
	\end{split}
\end{align}
Furthermore, observe that \eqref{eq:open:y} implies that for all $t\in(\tau, \tau+\varepsilon]$ it holds that
\begin{align}
\label{eq:u_t}
	\begin{split}
	u_t  
	& = y_{t-\tau} 
	\\
	& = 
		\int_0^{t-\tau} \! S_{t-\tau-s} \, F(y_s) \, d s + \int_0^{\tau} S_{t-s} \, F(x_s) \, d s + o_{t}\\
	& = 	
		\int_\tau^{t} S_{t-s} \, F(y_{s-\tau}) \, d s + \int_0^{\tau} S_{t-s} \, F(u_s) \, d s + o_{t}\\
	& = 
	\int_0^{t} S_{t-s} \, F(u_s) \, d s + o_{t}.
\end{split}
\end{align}
Next let $ I \subseteq [0, T]$ be the set given by $I=[0,\tau + \varepsilon]$.
Note that \eqref{eq:u_t}
and  \eqref{eq:open:u} establish that 
$I \supsetneq J=[0,\tau]$, $I\subseteq [0,T]$, $u \in C(I,V)$, $u|_J= x$, $ \forall \, t \in I \colon \int_0^t \|S_{t-s} \, F(u_s) \|_V \, ds < \infty$, and $ \forall \, t \in I \colon  u_t = \int_0^t S_{t-s} \, F(u_s) \, ds + o_t$. The proof of Lemma~\ref{lemma:open} is thus completed.
\end{proof}

\begin{theorem}
\label{thm:unique}
Let $(V, \left\|\cdot\right\|_V)$ 
be a separable $\mathbb{R}$-Banach space,
let $(W,\left\|\cdot\right\|_W)$ 
be an $\mathbb{R}$-Banach space,
let $T \in (0, \infty)$,  $F \in C(V, W)$, $o \in C([0,T],V)$,  
let $S \colon (0,T) \to L(W,V)$ be a $\mathcal{B}((0,T))\slash \mathcal{B}(L(W,V))$-measurable function, 
let $\mathcal{S} \colon [0,T] \to L(V)$ be a $\mathcal{B}([0,T])\slash \mathcal{B}(L(V))$-measurable function, 
let $\phi \colon (0,T) \to (0,\infty)$ be a function,
and assume for all $r \in [0,\infty)$, $t_1 \in [0,T)$, $t_2\in (0,T-t_1)$, 
$v \in V$ that 
$([0,T] \ni t  \mapsto \mathcal{S}_t v \in V) \in C([0,T],V)$,
$S_{t_1+t_2}=\mathcal{S}_{t_1} S_{t_2}$, 
$ \sup \big( \big\{\frac{\|F(v)-F(w)\|_W}{\|v-w\|_V} \colon v, w \in V, v \neq w, \|v\|_V + \|w\|_V \leq r\big\} \cup \{0\}  \big) + 
\inf_{\alpha \in (0,1)} \sup_{s \in (0,T), t \in (s,T)}  [s^{\alpha} (\|S_s\|_{L(W,V)} +
\|S_{t} - S_{s} \|_{L(W,V)}|\phi(t-s)|^{-1})] < \infty$, 
and $ \limsup_{t \searrow 0} \phi(t)=0$.
Then there exists a unique convex set $J \subseteq [0,T]$  with $0 \in I$ such that
\begin{enumerate}[(i)]
	\item \label{item:unique:1} 
there exists a unique $x \in C(J, V)$ which satisfies for all $t \in J$ that 
\begin{align}
\int_0^t \|S_{t-s} \, F(x_s) \|_V \, ds < \limsup_{s \nearrow \sup(J)} \left[\frac{1}{(T-s)}+ \|x_s\|_V\right]\!= \infty
\end{align}
and $x_t = \int_0^t S_{t-s} \, F(x_s) \, ds + o_t$ and 
	\item \label{item:unique:2}
for all convex sets $I \subseteq [0,T]$  and all $y \in C(I, V)$ with $I \supseteq J $,
$\forall \, t \in I \colon \int_0^t \|S_{t-s} \, F(y_s) \|_V \, ds < \infty$, and 
$ \forall \, t \in I \colon  y_t = \int_0^t S_{t-s} \, F(y_s) \, ds + o_t$ it holds that $x=y$.
\end{enumerate}
\end{theorem} 
\begin{proof}[Proof of Theorem~\ref{thm:unique}] 
Throughout this proof let $\mathcal{I}$ be the set 
given by
\begin{multline}
\mathcal{I} = \Bigg\{ I \subseteq [0, T] \colon \Bigg( \!\big(I \text{ is convex} \big) \!\wedge \!\big(0 \in I \big) \!\wedge\!  \bigg( \exists \, ! \, y\in C(I,V) \colon \Big[ \forall \, t \in I \colon \\
\Big(\! \big( \smallint\nolimits_0^t \|S_{t-s} \, F(y_s) \|_V \, ds < \infty \big)\! \wedge  \!\big( y_t = \smallint\nolimits_0^t S_{t-s} \, F(y_s) \, ds + o_t \big) \!\Big) \Big] \bigg)\! \Bigg)  \!\Bigg\}
\end{multline}
and let $J \subseteq [0, T]$ be the set given by 
\begin{equation}
\label{eq:J:first}
J=\bigcup_{I\in\mathcal{I} \cup\{\{0\}\}} I .
\end{equation}
Note that Theorem~\ref{thm:existence}  ensures
\begin{enumerate}[(a)]
	\item  that the set
	$\mathcal{I}$ is non-empty,
	\item\label{eq:J} that $J=\cup_{I\in\mathcal{I} } I$,
	and 
	\item\label{eq:J:0} that $J \neq \{0\}$.
\end{enumerate}
Moreover, observe that item~\eqref{eq:J} assures that for every $u \in J $ there exists a convex set $I \in \mathcal{I}$ with $u \in I$.   Hence, we obtain that for every $u \in J $ there exists a convex set  $I \subseteq [0, T]$ with $[0, u] \subseteq I$ such that 
there exists a unique function $y\in C(I,V)$ such that for all $t \in I$ it holds
 that 
 \begin{align}
 \int_0^t \|S_{t-s} \, F(y_s) \|_V \, ds < \infty \qquad \text{and} \qquad y_t = \int_0^t S_{t-s} \, F(y_s) \, ds + o_t.
 \end{align}
Corollary~\ref{cor:local:unique} therefore implies that 
there exists a unique  $x \in C(J,V)$ which satisfies for all $t \in J$  that 
$ \int_0^t \|S_{t-s} \, F(x_s) \|_V \, ds < \infty$  and 
\begin{equation}
\label{eq:x:exist}
x_t = \int_0^t S_{t-s} \, F(x_s) \, ds + o_t.
\end{equation}
In the next step we claim that
\begin{equation}
\label{eq:J:infinite}
\limsup_{s \nearrow \sup(J)} \left[\frac{1}{(T-s)}+ \|x_s\|_V\right]\!= \infty.
\end{equation}
We prove \eqref{eq:J:infinite} by contradiction. We thus assume that 
\begin{equation}
\label{eq:J:finite}
\limsup_{s \nearrow \sup(J)} \left[\frac{1}{(T-s)}+ \|x_s\|_V\right]\!< \infty.
\end{equation}
This assures that $\sup(J)<T$ and
\begin{align} \label{eq:unique:bound}
 \limsup_{s \nearrow \sup(J)}  \|x_s\|_V< \infty.
\end{align} 
Lemma~\ref{lemma:open}, item~\eqref{eq:J}, and item~\eqref{eq:J:0} hence establish that
\begin{equation}
\label{eq:J:open}
J=[0,\sup(J)).
\end{equation} 
Combining this, \eqref{eq:x:exist}, \eqref{eq:unique:bound}, and  Lemma~\ref{lemma:extension} ensures that $[0,\sup(J)] \in \mathcal{I}$. 
This and \eqref{eq:J:open} contradict to item~\eqref{eq:J}. This proves that
\begin{equation}
\label{eq:J:infinite:2}
\limsup_{s \nearrow \sup(J)} \left[\frac{1}{(T-s)}+ \|x_s\|_V\right]\!= \infty.
\end{equation}
Combining  item~\eqref{eq:J}, item~\eqref{eq:J:0}, and \eqref{eq:x:exist} hence
establishes items~\eqref{item:unique:1}--\eqref{item:unique:2}.
The proof of Theorem~\ref{thm:unique} is thus completed.
\end{proof}

\begin{cor}
\label{cor:unique}
Let $(V, \left\|\cdot\right\|_V)$ 
be a separable $\mathbb{R}$-Banach space,
let $(W,\left\|\cdot\right\|_W)$ 
be an $\mathbb{R}$-Banach space,
let $T \in (0, \infty)$,  $F \in C(V, W)$, $o \in C([0,T],V)$,  
let $S \colon (0,T) \to L(W,V)$ be a $\mathcal{B}((0,T))\slash \mathcal{B}(L(W,V))$-measurable function, 
let $\mathcal{S} \colon [0,T] \to L(V)$ be a $\mathcal{B}([0,T])\slash \mathcal{B}(L(V))$-measurable function, 
let $\phi \colon (0,T) \to (0,\infty)$ be a function,
and assume for all $r \in [0,\infty)$, $t_1 \in [0,T)$, $t_2\in (0,T-t_1)$, 
$v \in V$ that 
$([0,T] \ni t  \mapsto \mathcal{S}_t v \in V) \in C([0,T],V)$,
$S_{t_1+t_2}=\mathcal{S}_{t_1} S_{t_2}$, 
$ \sup \big( \big\{\frac{\|F(v)-F(w)\|_W}{\|v-w\|_V} \colon v, w \in V, v \neq w, \|v\|_V + \|w\|_V \leq r\big\} \cup \{0\}  \big) + 
\inf_{\alpha \in (0,1)} \sup_{s \in (0,T), t \in (s,T)} [ s^{\alpha} (\|S_s\|_{L(W,V)} +
\|S_{t} - S_{s} \|_{L(W,V)}|\phi(t-s)|^{-1})] < \infty$, 
and $ \limsup_{t \searrow 0} \phi(t)=0$.
Then there exists a unique convex set $J \subseteq [0,T]$ with $\{0\} \subsetneq J$ such that
\begin{enumerate}[(i)]
	\item \label{item:cor:unique:1} 
	there exists a unique  $x \in C(J, V)$  which satisfies for all $t \in J$ that 
	\begin{align}
	\int_0^t \|S_{t-s} \, F(x_s) \|_V \, ds < \limsup_{s \nearrow \sup(J)} \left[\frac{1}{(T-s)}+ \|x_s\|_V\right]\!= \infty
	\end{align}
	and $x_t = \int_0^t S_{t-s} \, F(x_s) \, ds + o_t$, 
	\item \label{item:cor:unique:2}
	for all convex sets $I \subseteq [0,T]$  and all  $y \in C(I, V)$ with $I \supseteq J $,
	$\forall \, t \in I \colon \int_0^t \|S_{t-s} \, F(y_s) \|_V \, ds < \infty$, and 
	$ \forall \, t \in I \colon  y_t = \int_0^t S_{t-s} \, F(y_s) \, ds + o_t$ it holds that $x=y$, and
	\item \label{item:cor:unique:last} for all  $y \in C(J, V)$   with $\forall \, t \in J \colon \int_0^t \|S_{t-s} \, F(y_s) \|_V \, ds < \infty$,  
	$ \forall \, t \in J \colon  y_t = \int_0^t S_{t-s} \, F(y_s) \, ds + o_t$, and $\limsup_{s \nearrow \sup(J)} \|y_s\|_V < \infty$  it holds that $J = [0, T]$.
\end{enumerate}
\end{cor}
\begin{proof}[Proof of Corollary~\ref{cor:unique}]
First, note that Theorem~\ref{thm:unique} ensures that there exists a unique convex set $J \subseteq [0,T]$ with $0 \in J$ which satisfies that
\begin{enumerate}[(a)]
	\item \label{item:proof:unique:1} 
	there exists a unique  $x \in C(J, V)$  which satisfies for all $t \in J$ that 
	\begin{align}
	\label{eq:infty}
	\int_0^t \|S_{t-s} \, F(x_s) \|_V \, ds < \limsup_{s \nearrow \sup(J)} \left[\frac{1}{(T-s)}+ \|x_s\|_V\right]\!= \infty
	\end{align}
	and $x_t = \int_0^t S_{t-s} \, F(x_s) \, ds + o_t$ and 
	\item \label{item:proof:unique:2}
	for all convex sets $I \subseteq [0,T]$  and all  $y \in C(I, V)$  with $I \supseteq J $,
	$\forall \, t \in I \colon \int_0^t \|S_{t-s} \, F(y_s) \|_V \, ds < \infty$, and 
	$ \forall \, t \in I \colon  y_t = \int_0^t S_{t-s} \, F(y_s) \, ds + o_t$ it holds that $x=y$.
\end{enumerate}	
Observe that  \eqref{eq:infty} demonstrates that
\begin{equation}
\{0\} \subsetneq J.
\end{equation} 
It thus remains to prove that for all  $y \in C(J, V)$   with $\forall \, t \in J \colon \int_0^t \|S_{t-s} \, F(y_s) \|_V $ $  ds < \infty$,  
$ \forall \, t \in J \colon  y_t = \int_0^t S_{t-s} \, F(y_s) \, ds + o_t$, and $\limsup_{s \nearrow \sup(J)} \|y_s\|_V < \infty$  it holds that 
\begin{equation}
\label{eq:J:global}
J = [0, T].
\end{equation}
We prove \eqref{eq:J:global} by contradiction. We thus assume that that there exists   $y \in C(J, V)$    which satisfies for all $t \in J$ that $ \int_0^t \|S_{t-s} \, F(y_s) \|_V \, ds < \infty$,  
$  y_t = \int_0^t S_{t-s} \, F(y_s) \, ds + o_t$, $\limsup_{s \nearrow \sup(J)} \|y_s\|_V < \infty$, and
\begin{equation}
\label{eq:J:noteq}
J \neq [0, T].
\end{equation}
Note that  the fact that $J$ is a convex set and item~\eqref{item:proof:unique:2} prove that $ y =x$. Combining this with the hypothesis that $\limsup_{s \nearrow \sup(J)} \|y_s\|_V < \infty$ 
 ensures that
\begin{equation}
\label{eq:x:bounded}
\limsup_{s \nearrow \sup(J)} \|x_s\|_V < \infty.
\end{equation}
This, \eqref{eq:J:noteq}, and 
item~\eqref{item:proof:unique:1}  imply that
\begin{equation}
\label{eq:J:neq}
J = [0, T).
\end{equation}
Lemma~\ref{lemma:extension} and \eqref{eq:x:bounded} hence assure that there exists  $z \in C( [0, T], V)$  such that for all $t \in [0, T]$ it holds that 
 \begin{align}
\int_0^t \|S_{t-s} \, F(z_s) \|_V \, ds < \infty \qquad \text{and} \qquad z_t = \int_0^t S_{t-s} \, F(z_s) \, ds + o_t.
\end{align}
Combining this and \eqref{eq:J:neq} contradicts to  item~\eqref{item:proof:unique:2}. The proof of Corollary~\ref{cor:unique} is thus completed.
\end{proof}

\bibliographystyle{acm}
\bibliography{bibfile} 
\end{document}